\documentclass{amsart}
\usepackage[foot]{amsaddr}
\usepackage{amssymb,latexsym,amsmath,amsfonts,graphicx,xcolor,pb-diagram}
\usepackage[mathscr]{eucal}
\usepackage{tikz, tikz-cd}
\usepackage{relsize}
\usepackage{slashed}

\newcommand{\mpar}{\par \medskip \par }
\newcommand{\Div}{\operatorname{div}}
\newcommand{\Image}{\operatorname{im}}
\newcommand{\Ker}{\operatorname{ker}}
\newcommand{\Curl}{\operatorname{curl}}
\newcommand{\Spt}{\operatorname{spt}}

\def\DirProd{\operatornamewithlimits{%
  \mathchoice{\vcenter{\hbox{\LARGE $\amalg$}}}
             {\vcenter{\hbox{\Large $\amalg$}}}
             {\mathrm{\amalg}}
             {\mathrm{\amalg}}}}
\def\DirSum{\operatornamewithlimits{%
  \mathchoice{\vcenter{\hbox{\huge $\oplus$}}}
             {\vcenter{\hbox{\Large $\oplus$}}}
             {\mathrm{\oplus}}
             {\mathrm{\oplus}}}}

\numberwithin{equation}{section}

\swapnumbers

\theoremstyle{definition}

\newtheorem{theorem}{Theorem}[section]
\newtheorem{lemma}[theorem]{Lemma}
\newtheorem{corollary}[theorem]{Corollary}
\newtheorem{proposition}[theorem]{Proposition}

\newtheorem{definition}[theorem]{Definition}

\newtheorem{remark}[theorem]{Remark}
\newtheorem{example}[theorem]{Example}

\makeindex

\begin{document}

\title{Helmholtz-Hodge Decomposition on Graphs}

\author{Peter March}

\begin{abstract}
We propose a definition of the curl of a vector field $X$ on a finite simple graph as the projection of $X$ onto the orthogonal complement of circulation-free vector fields, where a vector field is circulation-free provided its line integral around every simple circuit vanishes. We justify the definition by observing that $X$ and $\Curl X$ have the same circulation and $\Curl\circ\nabla = \Div\circ\Curl =0$. This shows the gradient, curl, and divergence operators form an exact sequence, in analogy with the classical case of vector fields on domains in $\mathbb{R}^3\negthinspace\negthinspace,$ and yields the Helmholtz-Hodge decomposition of a vector field on $G$ as the sum of a gradient, a curl, and a harmonic field. Along the way, we also prove analogues of the divergence theorem, Green's identities, and Helmholtz's theorem. A consequence of our definition is that the curl is a non-local operator, in sharp contrast to the classical case in $\mathbb{R}^3$ and existing notions of curl on a graph.
\end{abstract}

\address{Department of Mathematics\\
Rutgers University\\
Hill Center - Busch Campus\\
110 Frelinghuysen Road\\
Piscataway, NJ 08854-8019}

\email{march@math.rutgers.edu}

\maketitle

\section{Introduction}
It's part of the mathematical folklore that some ideas from calculus and differential geometry extend rather naturally from the continuous setting of manifolds to the discrete setting of graphs. This observation is pervasive in scientific computing and computer graphics e.g., \cite{A}, \cite{AFW}, \cite{C}, \cite {DDT}, \cite{GP}, \cite{L} where equations on structured graphs are meant to approximate equations on a manifold, and in discrete mathematics e.g., \cite{BCEG}, \cite{F}, \cite{LLY}, \cite{O}, \cite{SSWJ} where the graph is understood to be the primary object of study and exploring its geometry is an end in itself. This work is of the latter kind and focuses on the extent to which familiar notions  of multivariable calculus have analogues on a graph, in the absence of any special assumptions on the graph's structure.

\mpar
In the discrete setting there is a high degree of familiarity and consensus concerning the gradient of a function, the divergence of a vector field, and the Laplacian of a function, being the divergence of the gradient. However, this cannot be said of the other classic operator of multivariable calculus, the curl.  In fact, there isn't even consensus about the definition of a vector field on a graph, as some authors consider it to be a general function of the oriented edges of a graph while others insist it must be an anti-symmetric, or alternating, function of oriented edges. In addition, there are several notions of curl in use by mathematicians in differing circumstances. For example, \cite{L} considers the curl of a vector field to be a certain transformation from alternating functions on edges to alternating functions on triangles, which is quite natural from the perspective of the clique complex of the graph. On the other hand \cite{BCEG} considers the curl of a vector field, thought of as a general function of directed edges, to be the symmetrization of the vector field. In this case the curl is an operator from vector fields to vector fields such that curl of the gradient and divergence of the curl vanish. An early example of curl in a discrete setting is \cite{GH} which considers only planar triangular meshes.

\mpar
An antecedent of our work is the dissertation of Alexander Strang \cite{S} which explicitly considers line integrals of vector fields around cyclic subgraphs, rather than just triangles, and inspired our definition of curl. The two definitions are related in that they are couched in terms of integrals around cycles and lead to decompositions of vector fields. But they are distinct in that they apply to different definitions of vector fields, use different functional forms of the curl operator, and lead to different decompositions. 
\mpar
Our purpose here is to show that the notion of tangent graphs introduced in \cite{M} provides a convenient frame of reference for proposing a geometrically natural definition of curl and showing that every vector field on a graph is uniquely the sum of a gradient, a curl, and a harmonic vector field - an analogue of the Helmholtz-Hodge decomposition of vector fields on domains in $\mathbb{R}^3.$

\mpar
To start, we define the tangent graph, tangent bundle, and vector fields then introduce the  gradient, divergence, and Laplace operators: $\nabla, \Div$ and $\Delta$, respectively. This repeats some of the material presented in \cite{M} but we include it here to keep the article self-contained. 

\mpar
In the next section we provide short, conceptually clear proofs the divergence theorem, Green's identities, and Helmholtz's theorem which show the usefulness and pedagogical value of this approach. 

\mpar
In the final section, we introduce the circulation of a vector field and use it to define the curl operator as the projecton onto the orthogonal complement of the space of circulation-free vector fields. A notable consequence of this definition is that the curl is not a local operator since it is defined in terms of the solution set of a system of linear equations for the coefficients of vector fields on all cycles of $G.$

\mpar
With this definition in hand, it's straightforward to conclude that,
$$
0\xrightarrow{} \mathring{C}(G)\xrightarrow{\nabla}\mathcal{X}(G)\xrightarrow{\Curl}\mathcal{X}(G)\xrightarrow{\Div}\mathring{C}(G)\xrightarrow{}0
$$
is an exact sequence, where $\mathcal{X}(G)$ is the space of vector fields on $G$ and $\mathring{C}(G)$ is the space of functions on $G$ having zero average value. A standard result in linear algebra implies the orthogonal decomposition,
$$
\mathcal{X}(G) = \Image(\nabla)\oplus\Image(\Curl)\oplus\Ker(\nabla\circ\nabla^* + \Curl^*\circ\Curl).
$$
It follows from the definitions that $\nabla^*=\Div$ and that $\Curl$ is a self-adjoint projection, hence the rightmost term above is $\mathcal{H}(G)=\Ker(\nabla\circ\Div +\Curl).$ This is the space of harmonic vector fields, meaning vector fields that are both divergence-free and circulation-free, and yields the Helmholtz-Hodge decomposition. We show that,
\begin{align*}
|\mathcal{H}(G)| & = |V_G|-1,\,\,\text{and,}\\
|\Image(\Curl)| & =2(E_G|-|V_G|+1),
\end{align*}
so that the dimension of the image of the curl operator is twice the cyclomatic number of $G.$

\section{Tangent Graphs}
This section is devoted to background material. Much of this will be familiar except perhaps for the notion of the tangent graph which we use as an organizing principle. While the arguments are elementary we provide detailed proofs for completeness's sake.

\mpar
Let $G=(V_G, E_G)$ be a finite, simple graph having vertex set $V_G$ and edge set $E_G.$ If $\{i,j\}\in E_G$ then the ordered pair $ij=(i,j)$ represents the edge directed from $i$ to $j$ while the ordered pair $ji=(j,i)$ represents the edge directed from $j$ to $i$. The vertex set of the tangent graph $tG$ is the set of all directed edges and the edge set of $tG$ is the set of all pairs of directed edges where the endpoint of one directed edge is the basepoint of another directed edge. We think of the vertices of the tangent graph as tangent directions in $G$ and its edges as pairs of contiguous tangent directions. 

\mpar
The tangent graph turns out to be an oriented version of the line graph of $G$ in which each edge $\{i,j\}$ in $G$ determines an edge $\{ij, ji\}$ in $tG$ and each pair of incident edges $\{i,j\}, \{j,k\},\, k\neq i$ in $G$ determines a pair of edges $\{ij, jk\}, \{ji, kj\}$ in $tG$. As it happens, the tangent graph of $G$ determines the line graph of $G$ but not conversely.  In fact, the tangent graph is a strictly finer invariant than the line graph and the line graph is a minor of the tangent graph. (\cite{M} Proposition 2.12.3 and Remark 2.13.2)

\begin{definition}
Let $G=(V_G, E_G)$ be a finite, simple graph with adjacency matrix $A$. The \textit{tangent graph} $tG=(V_{tG}, E_{tG})$ has vertex set,
$$
V_{tG}=\{(i,j)\in V_G\times V_G\mid A(i,j)=1\}
$$
and edge set,
$$
E_{tG}=\{\{(i,j), (k,l)\}\in V_{tG}\times V_{tG}\mid j=k\,\, \text{or}\,\, i=l\}.
$$
\end{definition}

\begin{example}
Let $G$ be a triangle with an appended edge and let's compare its line graph with its tangent graph. It's straightforward to verify the line graph is two triangles joined along an edge but the structure of the tangent graph is not so obvious. Evidently, $tG$ has eight vertices - two for each of the four edges of $G.$ Every edge of $tG$ has the form $\{ij, ji\}$ or $\{ij, jk\}$ for distinct vertices $i,j,k.$ One can verify the depiction of $tG$ below by inspection. 

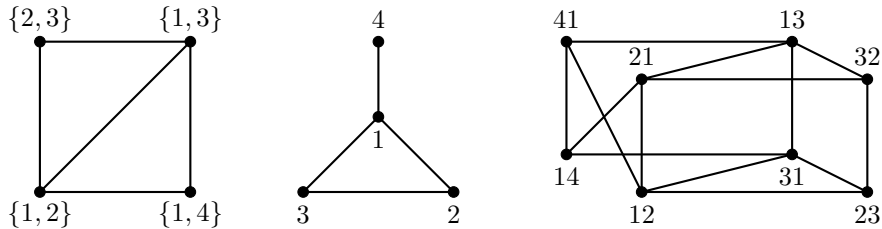
\begin{figure} [h]
\begin{tikzpicture}
\draw[fill=black] (0,0) circle (2pt);
\draw[fill=black] (0,2) circle (2pt);
\draw[fill=black] (2,0) circle (2pt);
\draw[fill=black] (2,2) circle (2pt);

\node at (0,-0.3) {$\{1,2\}$};
\node at (0, 2.3) {$\{2,3\}$};
\node at (2,-0.3) {$\{1,4\}$};
\node at (2,2.3) {$\{1,3\}$};

\draw[thick] (0,0) -- (2,0) -- (2,2) -- (0,0);
\draw[thick] (0,0) -- (0,2);
\draw[thick] (0,2) -- (2,2);

\draw[fill=black] (3.5,0) circle (2pt);
\draw[fill=black] (5.5,0) circle (2pt);
\draw[fill=black] (4.5,1) circle (2pt);
\draw[fill=black] (4.5,2) circle (2pt);

\node at (3.5,-0.3) {3};
\node at (5.5, -0.3) {2};
\node at (4.5, 0.7) {1};
\node at (4.5, 2.3) {4};

\draw[thick] (3.5,0) -- (5.5,0) -- (4.5,1) -- (3.5,0);
\draw[thick] (4.5,1) -- (4.5,2);

\draw[fill=black] (8,0) circle (2pt);
\draw[fill=black] (11,0) circle (2pt);
\draw[fill=black] (10,0.5) circle (2pt);
\draw[fill=black] (8,1.5) circle (2pt);
\draw[fill=black] (11,1.5) circle (2pt);
\draw[fill=black] (10,2) circle (2pt);
\draw[fill=black] (7,0.5) circle (2pt);
\draw[fill=black] (7,2) circle (2pt);

\node at (8,-0.3) {12};
\node at (11,-0.3) {23};
\node at (10,0.2) {31};
\node at (8,1.8) {21};
\node at (11,1.8) {32};
\node at (10,2.3) {13};
\node at (7,0.2) {14};
\node at (7, 2.3) {41};

\draw[thick] (8,0) -- (11,0) -- (10,0.5) -- (8,0) -- (8, 1.5) -- (11, 1.5) --  (10,2) -- (8, 1.5);
\draw[thick] (10, 0.5) -- (10, 2);
\draw[thick] (11, 0) -- (11, 1.5);
\draw[thick] (7, 0.5) -- (7, 2);
\draw[thick] (7, 2) -- (8, 0);
\draw[thick] (7, 2) -- (10, 2);
\draw[thick] (7, 0.5) -- (10, 0.5);
\draw[thick] (7, 0.5) -- (8, 1.5);
\end{tikzpicture}
\caption{A graph $G$ (\textit{center}) flanked by its line graph $lG$ (\textit{left}) and its tangent graph $tG$ (\textit{right}). Observe that $lG$ is obtained from $tG$ by contracting edges of the form $\{ij,ji\}$ and identifying their associated endpoints and incident edges.}
\end{figure}
\end{example}

\begin{proposition}
Let $\sigma\colon V_{tG}\to V_{tG}$ be the \textit{involution} $\sigma((i,j))=(j,i)$. Let $\pi, \pi_+\colon V_{tG}\to V_G$ be the \textit{projections} $\pi((i,j))=i$ and $\pi_+((i,j))=j.$ Then $\pi, \sigma,$ and $\pi_+=\pi\circ\sigma$ extend to graph homomorphisms.
\end{proposition}

\begin{proof}
A function $\phi\colon V_H\to V_K$ from the vertices of a graph $H$ to the vertices of a graph $K$ defines a homomorphism provided $\{i,j\}\in E_H$ implies $\{\phi(i), \phi(j)\}\in E_K.$ 

\mpar
To show $\pi$ is a homomorphism let $\{(i,j), (k,l)\}$ be an edge of $tG$ and suppose, say, that $j=k$. Then $\{\pi((i,j)), \pi((k,l))\} = \{i,k\} = \{i,j\}\in E_G.$ Alternatively, suppose $i=l$. Then $\{\pi((i,j)), \pi((k,l))\} = \{i,k\} = \{l,k\}\in E_G,$ and it follows that $\pi$ is a homomorphism. Note that this argument shows that the definition of edges in $tG$ precludes the possibility that $i=k,$ which is consistent with the fact that a simple graph has no self loops.

\mpar
We show $\sigma$ is a homomorphism by a similar argument. Suppose $\{(i,j), (k,l)\}\in V_{tG}$ and $j=k.$ Now $\sigma((i,j))=(j,i),$ and $ \sigma((k,l))=(l,k)$ and it follows that $\{(j,i), (l,k)\}\in E_{tG}$ because $j=k$. Suppose, on the other hand, that $i=l$. Then it follows from the definition that $\{(j,i), (l,k)\}\in E_{tG}.$ Since $\pi_+ = \sigma\circ\pi$, it is also a homorphism.
\end{proof}

\begin{remark}
It's useful to have notation for a vertex of the tangent graph that does not make explicit reference to the underlying edge or vertices. Specifically, if $u\in V_{tG}, \pi(u)=i,$ and $\pi_+(u)=j$ then $\{i, j\}\in E_G$ and we write $u=ij=(i,j).$ In this notation, $\pi(u)$ is the \textit{base point} of $u$ and $\pi_+(u)$ is the \textit{end point}. The vertex set and edge set of $tG$ are,
\begin{align*}
V_{tG}&=\{u\in V_G\times V_G\mid A(\pi(u), \pi_+(u))=1\},\\
E_{tG}& =\{\{u,v\}\mid u,v\in V_{tG},\, \pi_+(u)=\pi(v)\,\,\text{or}\,\,\pi_+(u)=\pi(v)\}.
\end{align*}
In the sequel, we almost always denote a vertex of the tangent graph generically by a letter like $u$ or as a concatenation of adjacent vertices like $ij$ but almost never by an ordered pair $(i,j).$
\end{remark} 

\begin{definition}
1. Let $C(G)$ be the space of real valued functions defined on $V_G$ and let $\mathring{C}(G)$ be the subspace of functions such that $\sum_{i\in V_G}\phi(i) = 0.$

\mpar
2. The functions $e_i, i\in V_G$, where,
$$
e_i(j)= \begin{cases}
1, & \text{$i=j,$}\\
0, & \text{$i\neq j,$}
\end{cases}
$$
form a basis of $C(G)$. Similarly, the functions, $e_u, u\in V_{tG}$, where,
$$
e_u(v)= \begin{cases}
1, & \text{$u=v,$}\\
0, & \text{$u\neq v,$}
\end{cases}
$$
form a basis of $C(tG).$ These spaces have natural inner products given by the rule $\langle e_i, e_j\rangle_{C(G)}=e_i(j)$ and $\langle e_u, e_v\rangle_{C(tG)} = e_u(v).$

\mpar
3. The \textit{tangent space} to $G$ at $i\in V_G$ is the vector space $T_i(G)=\langle e_u\mid \pi(u)=i\rangle $ where the angle brackets denote the space spanned by the indicated functions. Every $X_i\in T_i(G)$ has the form $X_i=\sum_{\pi(u)=i}X_i(u) e_u$ for some real numbers $X_i(u)$ called the \textit{coefficients} of $X_i$. $T_i(G)$ inherits an inner product from $C(tG)$ by the rule,
$$
\langle X_i, Y_i\rangle_{T_i(G)} =\sum_{\pi(u)=i}\sum_{\pi(v)=i}X_i(u)Y_i(v)\langle e_u, e_v\rangle_{C(tG)} =\sum_{\pi(u)=i}X_i(u)Y_i(u).
$$

\mpar
4. The \textit{tangent bundle} $T(G)$ of $G$ is the coproduct $T(G)=\DirProd_{i\in V_G}T_i(G).$ It inherits an adjacency relation from $G$ by saying $X_i$ and $X_j$ are adjacent in $T(G)$ provided $i$ and $j$ are adjacent in $G$. (Strictly speaking, we are conflating a vector in a tangent space with its canonical injection into the coproduct; a useful ambiguity which is unlikely to cause confusion).

\mpar
5. A \textit{vector field} $X$ on $G$ is a section of the tangent bundle; that is, a function $X\colon V_G\to T(G)$ where $X(i)\in T_i(G)$ for all $i\in V_G.$
Generally speaking, we write $X(i)=X_i$ for the vector values of the section and $X(u)=X_{\pi(u)}(u)$ for the cofficients of the vector field. Specifically, we always think of the coefficients of $X$ as a function in $C(tG)$ and observe that every such function is the set of coefficients of some vector field. Thus, we variously write,
$$
X=\sum_{u\in V_{tG}}X(u)e_u =\sum_{i\in V_G}\sum_{\pi(u)=i} X_i(u)e_u = \sum_{i\in V_G} X_i,
$$
keeping in mind our convention of conflating a tangent vector at a vertex with its canonical injection into the tangent bundle. 

\mpar
6. The space of all vector fields is denoted $\mathcal{X}(G).$ It follows from the definitions that $|\mathcal{X}(G)|=2|E_G|$ and $\mathcal{X}(G)\cong\DirSum_{i\in V_G}T_i(G)\cong C(tG).$
Note that $\mathcal{X}(G)$ inherits an inner product from $C(tG)$ by the rule,
\begin{align*}
\langle X,Y\rangle_{\mathcal{X}(G)} & = \sum_{u\in V_{tG}}\sum_{v\in V_{tG}}X(u)Y(v)\langle e_u, e_v\rangle_{C(tG)}\\
& = \sum_{u\in V_{tG}}X(u)Y(u) = \sum_{i\in V_G}\langle X_i, Y_i\rangle_{T_i(G)}.
\end{align*}

If $X,Y\in \mathcal{X}(G)$ we often write $\langle X_i, Y_i\rangle_{T_i(G)} =X_i\cdot Y_i$ so that, for example, 
$$
\langle X,Y\rangle_{\mathcal{X}(G)} = \sum_{i\in V_G} X_i\cdot Y_i.
$$
\end{definition}

The following proposition records the fact that the involution $\sigma\colon V_{tG}\to V_{tG}$ extends to an involution on $\mathcal{X}(G)$, leading to a $\mathbb{Z}_2$-grading of vector fields. This is a familiar fact stated in the framwork of tangent graphs.

\begin{proposition}
Let $\overline{u}=\sigma(u)$ and for every $X\in\mathcal{X}(G)$ let $\overline{X}$ be the vector field with coefficients $\overline{X}(u)=X(\overline{u}).$ We say $X$ is even or \textit{symmetric} provided $\overline{X}=X$ and it is odd or \textit{antisymmetric} provided $\overline{X}=-X.$ 

\mpar
The operators $s, a\colon\mathcal{X}(G)\to\mathcal{X}(G)$ defined by the formulas $sX =\tfrac{1}{2}(X+\overline{X})$ and $aX=\tfrac{1}{2}(X-\overline{X})$ are orthogonal projections satisfying $s+a=1.$ Let $\mathcal{X}^s(G) = \Image(s) = \Ker(a)$ and $\mathcal{X}(G)^a = \Image(a) =\Ker (s)$. Then, 
$$
\mathcal{X}(G)=\mathcal{X}^s(G)\oplus\mathcal{X}^a(G)
$$
is an orthogonal decomposition of vector fields into even and odd parts.
\end{proposition}
 
\begin{proof}
The calculations showing $s$ and $a$ are orthogonal projections are elementary and familiar but worth repeating. We have,
$$
 s^2X=\tfrac{1}{2} (s(X)+\overline{s(X)}))=\tfrac{1}{4}(X+\overline{X}) + \tfrac{1}{4}(\overline{X}+X) =\tfrac{1}{2}(X+\overline{X}) = sX,
$$
and $a^2=(1-s)^2 = 1-2s +s^2 = 1-s =a.$ To see that the decomposition is orthogonal, observe that,
\begin{align*}
2sX&=\negthickspace\sum_{u\in V_{tG}}\tfrac{1}{2}(X(u)+X(\overline{u}))(e_u+e_{\overline{u}}),\,\,\text{and}\\
2aX&=\negthickspace\sum_{u\in V_{tG}}\tfrac{1}{2}(X(u)-X(\overline{u}))(e_u-e_{\overline{u}}),
\end{align*}
so it's enough to show that, 
$$
\langle e_u+e_{\overline{u}}\,, e_v-e_{\overline{v}}\rangle_{\mathcal{X}(G)}=  e_u(v)-e_u(\overline{v})+e_{\overline{u}}(v) -e_{\overline{u}}(\overline{v})= 0
$$
for all $u,v\in V_{tG}.$ Now, the terms above vanish identically if $u\notin\{v, \overline{v}\}.$ But, by inspection, the sum of the terms vanishes if either $u=v$ or $u=\overline{v}.$
\end{proof} 

\begin{remark}
Some authors insist that a vector field on a graph should be odd or antisymmetric, in analogy with the case of manifolds. We prefer to admit the possibility of even or symmetric vector fields and, more importantly, vector fields of no particular parity.
\end{remark}

The following material is again quite familiar but it is restated in the framework of tangent graphs. It partly overlaps with some introductory material in \cite{M} but is repeated here for the reader's convenience.

\begin{proposition}
1. Let $d\colon C(G)\to C(tG)$ be the operator,
$$
d\phi(u)=\phi(\pi_+(u))-\phi(\pi(u))
$$ 
and $\pi,\pi_+\colon C(tG)\to C(G)$ be the operators,
$$
\pi f(i) =\negthickspace\sum_{\pi(u)=i}f(u)\,\,\text{and}\,\, \pi_+f(i) = \negthickspace\sum_{\pi+(u)=i}f(u).
$$
Then $d$ and $(\pi_+-\pi)$ are adjoint operators in the sense that,
$$
\langle d\phi, f\rangle_{C(tG)} =\langle \phi, \pi_+f -\pi f\rangle_{C(G)}.
$$

\mpar
2. The \textit{gradient} is the operator $\nabla \colon C(G)\to\mathcal{X}(G)$ defined by the formula,
$$
\nabla\phi=\sum_{u\in V_{tG}}d\phi(u)e_u,
$$
and the \textit{divergence} is the operator $\Div\colon\mathcal{X}(G)\to C(G)$ defined by the formula,
$$
\Div X = \sum_{i\in V_G} (\pi_+X(i)-\pi X(i)) e_i = \sum_{i\in V_G}\sum_{\pi(u)=i}(X(\overline{u})-X(u))e_i.
$$
Then $\nabla$ and $\Div$ are adjoint operators in the sense that,
$$
\langle\nabla\phi, X\rangle_{\mathcal{X}(G)}=\langle\phi, \Div X\rangle_{C(G)}.
$$

3. The \textit{Laplacian} is the operator $\Delta=\Div\circ\nabla.$ It is a non-negative, self adjoint operator on $C(G)$ given by the formula,
$$
\Delta\phi(i)=-2\sum_{\pi(u)=i}d\phi(u).$$
\mpar

4. Recall that every linear operator $L\colon C(G)\to C(G)$ has the form,
$$
L\phi(i)=\sum_{j\in V_G}L(i,j)\phi(j).
$$
We say $L$ is a \textit{first order differential operator} provided (1) $L\phi=0$ for every function $\phi$ that is constant on each connected component of $G$ and (2) $L(i,j)=0$ if $d(i,j)\geq 2$ where $d(i,j)$ is the length of the shortest path between $i$ and $j$. 

\mpar
Every vector field $X\in\mathcal{X}(G)$ defines a first order differential operator by the rule,
$$
X\phi(i)=\sum_{\pi(u)=i}X(u)d\phi(u)
$$
and every first order differential operator is of this form for some vector field $X$. In particular, $\Delta$ is the first order differential operator associated to the constant vector field $X = -2$. 

\mpar
5. For every $\nu\in C(G)$ let $m(\nu)$ be the operator $m(\nu)\phi(i) = \nu(i)\phi(i).$ Let $X^*$ be the adjoint of $X$, meaning $\langle X\phi, \psi\rangle_{C(G)} = \langle\phi, X^*\psi\rangle_{C(G)}.$ Then,
$$
X^*= \overline{X} +m(\Div X(i)).
$$ 
Thus, $X$ is self-adjoint if and only if $\overline{X} = X$ and $X$ is skew-adjoint if and only if both $\overline{X} = -X$ and $\Div X = 0.$ 
\end{proposition}

\begin{proof}
The proofs are straightforward calculations but we include them here to gain familiarity with the basic ideas and notation. Regarding item 1 we have,
\begin{align*}
\langle d\phi, f\rangle_{C(tG)} & = \sum_{u\in V_{tG}}d\phi(u) f(u)\\
& = \sum_{u\in V_{tG}}\phi(\pi_+(u)) f(u) - \sum_{u\in V_{tG}}\phi(\pi(u))f(u)\\
& = \sum_{i\in V_G}\phi(i)\negthickspace\sum_{\pi_+(u)=i}f(u) - \negthickspace\sum_{i\in V_G}\phi(i)\sum_{\pi(u)=i}f(u) \\
& = \langle\phi,\pi_+ f-\pi f\rangle_{C(G)}.
\end{align*}

The assertion that the divergence and gradient are adjoints of one another is just a restatement of item 1 in the language of vector fields. It follows that,
$$
\langle\Delta\phi, \psi\rangle =\langle\Div\circ\nabla\phi, \psi\rangle =\langle\nabla\phi,\nabla\psi\rangle=\langle \phi, \Div\circ\nabla\phi\rangle=\langle\phi,\Delta\psi\rangle,
$$
hence $\Delta$ is self adjoint. Since,
$$
\langle\Delta\phi, \phi\rangle_{C(G)} = \langle\nabla\phi, \nabla\phi\rangle_{\mathcal{X}(G)} \geq 0,
$$
it is non-negative. We have,
$$
\Delta\phi(i) = \Div\left(\nabla\phi\right) = \sum_{\pi(u)=i}\left(d\phi(\overline{u})- d\phi(u)\right)= -2\sum_{\pi(u)=i}d\phi(u),
$$
since $d\phi(ji)=-d\phi(ij).$ To calculate the adjoint of $X$ observe,
\begin{align*}
\langle X\phi, \psi\rangle_{C(G)} & = \sum_{i\in V_G} \psi(i) X\phi(i)\\
& = \sum_{i\in V_G} \psi(i)\negthickspace\sum_{\pi(u)=i} X(u) d\phi(u)\\
& = \sum_{u\in V_{tG}} \psi(\pi(u))\, X(u) d\phi(u)\\
& = \langle\nabla\phi, \psi X\rangle_{\mathcal{X}(G)}= \langle\phi, \Div(\psi X)\rangle_{C(G)},
\end{align*}
where $\psi X$ is the vector field with coefficients $(\psi X)(u) =  \psi(\pi(u)) X(u).$
Thus,
\begin{align*}
X^*\psi(i) & = \Div (\psi X)(i)\\
& = \sum_{\pi(u)=i}\left[\psi X(\overline{u}) -\psi X(u)\right]\\
& = \sum_{\pi(u)=i}\left[\psi(\pi_+(u)) X(\overline{u}) - \psi(i) X(u) \pm \psi(i) X(\overline{u})\right]\\
& = \sum_{\pi(u)=i} X(\overline{u})d\psi(u) + \psi(i)\negthickspace \sum_{\pi(u)=i}\left[X(\overline{u}) - X(u)\right]\\
& = \overline{X}\psi(i) + \psi(i)\Div X(i).
\end{align*}

It's easy to see that if $\overline{X}=X$ then $\Div X\equiv 0$ so this symmetry condition is necessary and sufficient for self adjointness. The conditions for skew adjointness follow from the definitions.
\end{proof}

\section{Gradient, Divergence, Laplacian}
In this section we prove analogues of some of the classical theorems of multivariable calculus involving $\nabla, \Div,$ and $\Delta.$ Strictly speaking, only Helmholtz's theorem is needed in the sequel. To be clear, all the results presented here are part of the mathematical folklore. But they are interesting in their own right, show the value of tangent graphs as a conceptual framework, and set the stage for relating the gradient, divergence, and Laplacian to the curl operator, introduced in the next section.

\mpar
For simplicity's sake, we assume from now on that $G$ is connected as the general results follow from this case.

\mpar
The divergence theorem states, for a sufficiently regular Euclidean domain $D$ and sufficiently regular vector field $X $, that,
$$
\int_D\Div X(x)dx= \int_{\partial D} n_D\negthinspace\cdot\negthinspace X(y)\sigma (dy)
$$
where $dx$ is Lebesgue measure, $\sigma (dy)$ is surface measure, and $n_D$ is the normal vector field. This is a conservation law stating the total flux of $X$ in the domain equals the integral of the normal component of $X$ around the boundary. 

\mpar
If $X$ is now a vector field on $G$ and we interpret the coefficient $X(u)$ as the rate at which a unit of substance moves across the edge from $\pi(u)$ to $\pi_+(u)$ then $\Div X(i)$ is the flux of $X$ at $i$, being the difference between the rate at which the substance enters $i$ minus the rate at which it leaves $i$. So, if $H$ is a subgraph of $G$ then $\sum_{j\in V_H}\Div X(j)$ is clearly the total flux of $X$ in $H$. Thus, to formulate the divergence theorem on $G$ we need to define the boundary of $H$, denoted $\partial H$, thought of as a subgraph of $G,$ as well as the normal to $H,$ denoted $n_H,$ thought of as a vector field on $\partial H$.

\begin{definition}
1. Let  $H$ be a subgraph of $G$. The \textit{boundary} of $H$ in $G$ is the graph $\index{$\partial H$}\partial H$ whose edge set is
$$
E_{\partial H} = \{\,\{i,j\}\in E_G\mid i\in V_H,\, j\notin V_H\,\},
$$
and whose vertex set $V_{\partial H}$ is the set of vertices of the edges in $E_{\partial H}$. Specifically, let,
$$ 
V_{\partial H}^- = \{ i\in V_H\mid\exists j\in V_G\setminus V_H\text{ such that } \{i,j\}\in E_{\partial H}\}
$$
and
$$V_{\partial H}^+ = \{ i\in V_G\setminus V_H\mid\exists j\in V_H\text{ such that } \{i,j\}\in E_{\partial H}\}.$$

\mpar
Then, $V_{\partial H}^+\cap V_{\partial H}^- = \emptyset$ and $V_{\partial H}^+\cup V_{\partial H}^-=V_{\partial H}.$ 

\mpar
2. The \textit{normal vector field} of $H$ in $G$ is the vector field $n_H\in\mathcal{X}(\partial H)$ such that 
$$
n_H(u) = \begin{cases}\phantom{-}1, & \text{if } \pi(u)\in V_{\partial H}^+\\ -1, & \text{if } \pi(u)\in V_{\partial H}^-\\ \phantom{-}0, & \text{else.}\end{cases}
$$
\end{definition}

\begin{remark}
A few remarks help to clarify the definition. 1. $\partial H$ is the bipartite subgraph of $G$ induced by the partition $V_{\partial H}= V_{\partial H}^+\cup V_{\partial H}^-.$ 

\mpar
2. $n_H$ is the inward normal since $n_H(u) = 1$ if and only if $\pi_+(u)\in V_H.$  

\mpar
3. An elementary calculation shows that $\nabla 1_H(u) = d 1_H(u)= n_H(u),$ if $u\in V_{t\partial H}$ and $\nabla 1_H(u) = 0,$ otherwise. In fact, this calculation motivates the definitions above and is the key observation used in the proof of the divergence theorem. 

\mpar
4. Since $E_H\cap E_{\partial H}=\emptyset$, the edge structure of $H$ plays no role in the definition of $\partial H$. Consequently, $n_H=n_{H'}$ if and only if $V_H=V_{H'}.$ Perhaps a more accurate notation for the normal vector field might be $n_{V_H},$ but this seems overly pedantic.

\mpar
5. It's natural to ask if this notion of boundary is homological in nature, meaning is the boundary of $\partial H$ empty? The answer is no, in general, if the boundary of $\partial H$ is understood to be its boundary relative to $G, $ as simple examples show. On the other hand, the answer is yes if the boundary of $\partial H$ is understood to be its boundary relative to $\overline{H} = H\cup \partial H.$
\end{remark}

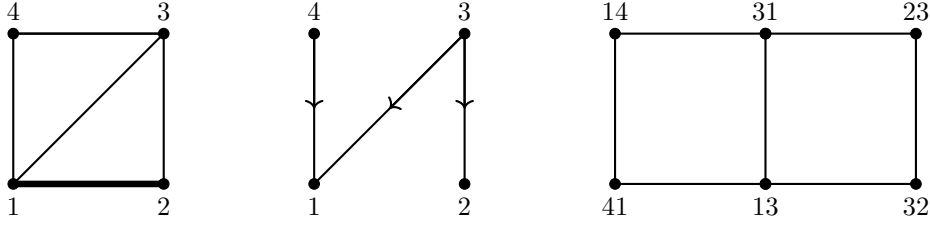
\begin{figure} [h]
\begin{tikzpicture} 
\draw[fill=black] (0,0) circle (2pt);
\draw[fill=black] (0,2) circle (2pt);
\draw[fill=black] (2,0) circle (2pt);
\draw[fill=black] (2,2) circle (2pt);

\node at (0,-0.3) {1};
\node at (0, 2.3) {4};
\node at (2,-0.3) {2};
\node at (2,2.3) {3};

\draw[line width=2.5 pt] (0,0) -- (2,0);
\draw[thick] (2,0) -- (2,2) -- (0,2) -- (0,0);
\draw[thick] (0,2) -- (2,2);
\draw[thick] (0,0) -- (2,2);

\draw[fill=black] (4,0) circle (2pt);
\draw[fill=black] (4,2) circle (2pt);
\draw[fill=black] (6,0) circle (2pt);
\draw[fill=black] (6,2) circle (2pt); 

\node at (4,-0.3) {1};
\node at (4, 2.3) {4};
\node at (6,-0.3) {2};
\node at (6,2.3) {3};

\draw[thick] (6,0) -- (6,2);
\draw[thick] (4,0) -- (4,2);
\draw[thick] (4,0) -- (6,2);

\draw[thick,->] (6,2) -- (6,1);
\draw[thick,->] (4,2) -- (4,1);
\draw[thick,->] (6,2) -- (5,1);

\draw[fill=black] (10,0) circle (2pt);
\draw[fill=black] (10,2) circle (2pt);
\draw[fill=black] (12,0) circle (2pt);
\draw[fill=black] (12,2) circle (2pt);

\draw[fill=black] (8,0) circle (2pt);
\draw[fill=black] (8,2) circle (2pt);
\draw[fill=black] (10,0) circle (2pt);
\draw[fill=black] (10,2) circle (2pt);

\node at (8,-0.3) {41};
\node at (10, -0.3) {13};
\node at (12,-0.3) {32};
\node at (8,2.3) {14};
\node at (10, 2.3) {31};
\node at (12,2.3) {23};

\draw[thick] (8,0) -- (10,0) -- (12,0) -- (12,2) -- (10,2) -- (8,2) --(8,0);
\draw[thick] (10,0) -- (10,2);

\end{tikzpicture}
\caption{From left to right the graphs are $G,\, \partial H,$ and $t\partial H.$ The subgraph $H$ is shown inside $G$ as the thick edge $\{1,2\}.$ Arrowheads on the edges of $\partial H$ indicate the direction in which $n_H = +1.$ }
\end{figure}
\begin{example}
Let $G$ be a rectangle with vertices $\{1,2,3,4\}$ plus a diagonal edge, say, $\{1,3\}.$ (See Figure 2). Let $H$ be the subgraph consisting of a non-diagonal edge, say $H=\{1,2\}.$ Then $\partial H$ is the graph having vertex subsets $V_{\partial H}^- = \{1,2\}$ and $V_{\partial H}^+ = \{3,4\}$ and edge set $H_{\partial H} = \{ \{1,2\}, \{1,3\}, \{2,3\}\},$ as one can verify by inspection. Note that the vertex set of the tangent graph of $\partial H$ is $V_{t\partial H}=\{13, 31, 14, 41, 23, 32\}.$ It's an easy exercise to verify the edges of $t\partial H$ are accurately depicted below.
\end{example}

\begin{proposition} 
(Divergence Theorem) Let $H$ be a subgraph of $G$ and let $X$ be a vector field on $G$. Then,
$$
\sum_{i\in V_H^{\phantom{-}}}\Div X(i) \,= \negthickspace\sum_{j\in V_{\partial H}}^{\phantom{-}} n_H\negthickspace\cdot\negthickspace X(j)\,= \negthickspace\sum_{j\in V_{\partial H}^-}\Div_{\partial H} X(j),
$$
where $\Div_{\partial H}\colon\mathcal{X}(\partial H)\to C(\partial H)$ is the divergence operator associated with the boundary graph.
\end{proposition}

\begin{proof}
The proof is a simple consequence of the mutual adjointness of the gradient and divergence. Let $1_H$ be the indicator function of $V_H$ and observe that $d 1_H(u) = n_H(u)$ on $V_{t\partial H}$ and $d 1_H(u)=0$ on $V_{tG}\setminus V_{t\partial H}$. Noting that $X$ restricts to a vector field on $\partial H,$ we have,
\begin{align*}
\sum_{j\in V_H}\Div X(j) & = \sum_{i\in V_G} 1_H(i)\Div X(i) = \langle 1_H, \Div X\rangle_{C(G)}\\
& =\langle\nabla 1_H, X\rangle_{\mathcal{X}(G)} = \sum_{u\in V_{t\partial H}}d 1_H(u) X(u)\\
& = \sum_{j\in V_{\partial H}}\sum_{\pi(u)=j} d 1_H(u)X(u) = \sum_{j\in V_{\partial H}} n_H\negthickspace\cdot\negthickspace X(j).  
\end{align*}
Observe that each directed edge $u\in V_{\partial H}$ contributes a term of the form $\pm X(u),$ where we have a plus sign if $\pi(u)\in V_{\partial H}^+$ and a minus sign if $\pi(u)\in V_{\partial H}^-.$ Thus,
$$
\sum_{j\in V_{\partial H}}^{\phantom{-}} n_H\negthickspace\cdot\negthickspace X(j)\, =\negthickspace\negthickspace\sum_{j\in V_{t\partial H}^-}\,\sum_{\substack{u\in V_{t\partial H}\\ \pi(u)=j}}\left(X(\overline{u})- X(u)\right),
$$
which we recognize as $\Div_{\partial H} X,$ the divergence of the restriction of $X,$ integrated over $V_{\partial H}^-.$
\end{proof}

\begin{remark}
1. Because the divergence of a symmetric vector field vanishes identically, the boundary term in the divergence theorem depends only on the asymmetric part of $X,$ thought of as a vector field on $\partial H.$
\mpar
2. When $H$ is an isolated vertex $i_0,$ that is a subgraph of $G$ with one vertex and no edges, then $\partial H$ is the union of all edges of $G$ having $i_0$ as a vertex. Thus,
$$
\Div X(i_0) = \sum_{i\in V_H}\Div X(i) = \sum_{j\in V_{\partial H}} n_H\negthickspace\cdot\negthickspace X(j) = \sum_{\pi (u)=i_0} (X(\overline{u}) - X(u)).  
$$
So, in this case the divergence theorem reduces to the definition of divergence. 

\mpar
3. Consider $V_{\partial H}^-$ as a subgraph of $\partial H$ having vertices but no edges. The divergence theorem applies to the pair $V_{\partial H}^-\subset\partial H$ appearing in the boundary sum and one can apply the divergence theorem to this pair to see if there is an advantage to be gained by iteration. However, $\partial V_{\partial H}^-=\partial H$ so this happens not to be the case.
\end{remark}

\begin{corollary}
(Green's Theorem)  Let $\Delta_{\partial H}$ be the Laplacian of $\partial H.$ Then,
$$
\sum_{i\in V_H}\Delta\phi(i)=\negthickspace\negthickspace\sum_{j\in V_{\partial H}}n_H\negthickspace\cdot\negthickspace\nabla\phi(j) =\negthickspace\sum_{j\in V_{t\partial H}^-}\Delta_{\partial H}\phi(j).
$$
\end{corollary}

\begin{proof}
Observe that $\nabla\phi\in\mathcal{X}(G)$ restricts to the vector field $\nabla_{\partial H}\phi\in\mathcal{X}(\partial H),$ hence,
$$
\sum_{j\in V_{\partial H}}n_H\negthickspace\cdot\negthickspace\nabla\phi(j) = \negthickspace\sum_{j\in V_{\partial H}^-}\Div_{\partial H}\nabla_{\partial H}\phi(j) = \negthickspace\sum_{j\in V_{\partial H}^-}\Delta_{\partial H}\phi(j).
$$

According to Remark 3.5.3 above, applying the Green's theorem to the pair $V_{\partial H}^-\subset \partial H$ does not improve the result.
\end{proof}

The next result is an extention of Green's theorem to all first order vector fields.

\begin{theorem}
Let $H$ be a subgraph of $G$ and $X \in\mathcal{X}(G)$ be a vector field. Then,
\begin{align*}
\sum_{i\in V_H} X\phi(i) & =\sum_{i\in V_H}\phi(i)\Div X(i) + \sum_{j\in V_{\partial H}}n_H\negthickspace\cdot\negthinspace\phi\overline{X}\\
& = \sum_{i\in V_H}\phi(i)\Div X(i) + \sum_{j\in V_{\partial H}^-}\Div_{\partial H}\phi\overline{X},
\end{align*}
where $\phi\overline{X}$ is the vector field with coefficients $(\phi\overline{X})(u) =\phi(\pi(u))X(\overline{u}).$ In particular, when $X\equiv -2$ then $X=\Delta,\, \Div X = 0$ and, 
$$
\sum_{j\in V_{\partial H}}n_H\negthinspace\cdot\negthinspace\phi\overline{X}=\sum_{j\in V_{\partial H}}n_H\negthinspace\cdot\negthinspace\nabla\phi(j),
$$
which reduces to Green's theorem.
\end{theorem}

\begin{proof}
According to Proposition 1.7.5 we have,
\begin{align*}
\sum_{i\in V_H} X\phi(i) & =\sum_{i\in V_G} X\phi(i) = \langle 1_H, X\phi\rangle_{C(G)}\\
& = \langle X^* 1_H, \phi\rangle_{C(G)}\\
& = \langle\overline{X}1_H, \phi\rangle_{C(G)}\, + \,\langle 1_H \Div X, \phi\rangle_{C(G)}\\
& = I + II.
\end{align*}
Evidently, $II = \sum_{i\in V_H}\phi(i) \Div X(i)$ and,
\begin{align*}
I  & = \sum_{i\in V_G} \phi(i) \negthickspace\sum_{\pi(u)=i}\overline{X}(u) d1_H(u)\\
&  = \sum_{u\in V_{t\partial H}} \negthickspace\negthickspace n_H(u)\phi(\pi(u))\overline{X}(u)\\
& = \sum_{j\in V_{\partial H}}\negthickspace\negthickspace n_H\negthickspace\cdot\negthinspace \phi\overline{X}(j)\\
&  =\sum_{j\in V_{\partial H}^-}\negthickspace\negthickspace\Div_{\partial H}\phi\overline{X}(j),
\end{align*}
according to the divergence theorem. It remains to evaluate this last sum when $X\equiv -2.$ We have,
$$
\sum_{j\in V_{\partial H}^-}\Div_{\partial H}\phi\overline{X}(j) = \sum_{j\in V_{\partial H}^-}\negthickspace\negthickspace -2\negthickspace\negthickspace\sum_{\substack{u\in V_{t\partial H}\\ \pi(u)=j}} [\phi(\pi(\overline{u}))-\phi(\pi(u))] = \sum_{j\in V_{\partial H}^-}\Delta_{\partial H}\phi(j),
$$
which is Green's theorem.
\end{proof}

\begin{remark}
This result has a natural physical interpretation in terms of fluid flow in a passive, spatial network of reservoirs and pipes modelled by the vertices and edges, respectively, of $G$. Assume that resevoirs are cylinders and we indicate the volume of fluid in resevoir $i$ by the height of the fluid surface $\phi(i).$ The rate at which fluid is passes through the one-way pipe $(i, j) $ from reservoir $i$ to reservoir $j$ is proportional to the height difference $\phi(j)-\phi(i),$ or the \textit{head}, with a rate constant $X(ij)$ determined by the dimensions of the cylinders and pipes. Then, the instantaneous rate of change of the volume of fluid at reservoir $i$ is $X\phi(i)$ and $\sum_{i\in V_H}X\phi(i)$ is the instantaneous rate of change of total fluid volume in the sub-network modelled by $V_H$. The theorem states that the volume rate of change of the sub-network that is not capured by calculating the net flux at the individual reservoirs of $V_H$ is accounted for by flow through the boundary pipes of the sub-network.
\end{remark}

Next we present a well-known but key lemma describing the kernel and image of the gradient, divergence, and Laplacian. It is used in the proof of Green's identities as well as the proof of Helmholtz's theorem.
\begin{lemma}
1. The kernel of $\Delta$ is the one dimensional space of constant functions and the image of $\Delta$ is $\mathring{C}(G)$. Hence, $\Delta\colon\mathring{C}(G)\to\mathring{C}(G)$
is an isomorphism whose inverse is denoted by $\Delta^{-1}$. 

\mpar
2. The kernel of $\nabla$ is also the one dimensional space of constant functions and $\nabla\colon\mathring{C}(G)\to\Image(\nabla)$ is an isomorphism. 

\mpar
3. We have $\Image(\Div)=\mathring{C}(G)$ and $\Ker(\Div)=\Image(\nabla)^{\perp}$.
\end{lemma}

\begin{proof}
Suppose $\Delta\phi$ = 0. Then,
$$
0 = \sum_{i\in V_G}\phi (i)\Delta\phi (i) = \sum_{i\in V_G} |\nabla\phi (i)|^2.
$$
Therefore $\nabla\phi = 0,$ hence $\phi$ is constant since $G$ is connected. Thus, $\Ker(\nabla)\cong\mathbb{R}$ and the restriction of the gradient to $\mathring{C}(G)$ is an isomorphism onto its image.

\mpar
Next, let $\psi = 1$ be the constant function. Then,
$$
\sum_{i\in V_G}\Delta\phi (i) =\sum_{i\in V_G} \psi (i)\Delta\phi (i) = \sum_{i\in V_G}\nabla\psi (i)\nabla\phi (i) = 0
$$
hence the image of $\Delta$ lies inside $\mathring{C}(G)$. By the rank-nullity theorem, $|C(G)|= |\Ker(\Delta)| + |\Image(\Delta)|.$ Hence, $|\Image\Delta| = |C(G)| -1$ and therefore $\Image (\Delta)=\mathring{C}(G).$ Since the restriction of $\Delta$ to $\mathring{C}(G)$ is both injective and surjective, it is invertible with inverse $\Delta^{-1}.$ 

\mpar
Again, let $\psi = 1$ be the constant function. Then,
$$
\sum_{i\in V_G}\Div X(i) =\sum_{i\in V_G}\psi (i)\Div X(i) = \sum_{i\in V_G}\nabla\psi (i)\negthinspace\cdot\negthinspace X(i) = 0,
$$
hence $\Image (\Div)\subset\mathring{C}(G).$ On the other hand, suppose $\phi\in\mathring{C}(G)$ and let,
$$
X=\nabla\circ\Delta^{-1}\phi.
$$ 
Then $\Div X = \Div\circ\,\nabla\circ\Delta^{-1}\phi=\phi$ hence $\Image (\Div) = \mathring{C}(G).$ 

\mpar
Finally, note that $\Div X = 0$ if and only if for all $\phi\in C(G),$
$$
0=\sum_{i\in V_G}\phi (i) \Div X(i) = \sum_{i\in V_G}\nabla \phi (i) X(i),
$$
and this is equivalent to saying $X\in\Ker(\Div)$ if and only if $X\in\Image (\nabla)^{\perp}$.
\end{proof}

\begin{definition}
1. For every $\phi\in C(G)$ and any subgraph $H$ of $G$ let,
$$
\overline{\phi}_H=\frac{1}{|V_H|}\sum_{i\in V_H}\phi(i)
$$
be the average of $\phi$ over $H$.

\mpar
2. For every $i\in V_G,$ Green's function with pole at $i,$ denoted $G_i,$ is defined by the formula,
$$
G_i(j)=\Delta^{-1}\left( e_i-\frac{1}{|V_G|}\right)(j).
$$
\end{definition}

\begin{proposition}
(Green's Identities) Let $H$ be a subgraph of G and $\phi, \psi\in C(G)$. Then we have Green's first identity,
$$
\sum_{j\in V_H} [\psi (j)\Delta\phi (j) - \nabla\psi\negthinspace\cdot\negthinspace\nabla\phi (j)]  = \,\sum_{j\in V_{\partial H}} n_H\negthinspace\cdot\negthinspace\psi\nabla\phi (j),
$$
second identity,
$$
\sum_{j\in H}[\psi(j)\Delta\phi(j) - \phi(j)\Delta\psi(j)]  = \,\sum_{j\in V_{\partial H}}[\psi(j)\, n_H\negthickspace\cdot\negthickspace\nabla\phi(j) - \phi(j)\, n_H\negthickspace\cdot\negthinspace\nabla\psi(j)],
$$
and third identity,
$$
\sum_{j\in V_H}G_i(j)\Delta\phi(j) = \,(\phi(i)1_H(i) - \frac{|V_H|}{|V_G|}\,\overline{\phi}_H)\\ + \sum_{j\in V_{\partial H}}[G_i(j)\, n_H\cdot\nabla\phi(j)-\phi(j)\, n_H\cdot\nabla G_i(j)].
$$
In particular, if $i\in V_H$ and $\overline{\phi}_H=0$ then,
$$
\sum_{j\in V_H}G_i(j)\Delta\phi(j) = \phi(i)+\sum_{j\in V_{\partial H}}[G_i(j)\, n_H\negthickspace\cdot\negthickspace\nabla\phi(j)-\phi(j)\,n_H\negthickspace\cdot\negthickspace\nabla G_i(j)].
$$
\end{proposition}

\begin{proof}
We prove the identities in order. Let $X=\psi\nabla\phi$. Then,
\begin{align*}
\Div X(i) & =\phantom{-}\sum_{\pi(u)=i}[X(\overline{u}) - X(u)]\\ 
& = \phantom{-}\sum_{\pi(u)=i}[\psi(\pi_+(u)) d\phi(\overline{u}) - \psi(i) d\phi(u))]\\ 
& =  -\sum_{\pi(u)=i}[\psi(\pi_+(u))+\psi(i)] d\phi(u)\pm\psi(i) d\phi(u)\\
& =  -\sum_{\pi(u)=i}[d\psi(u)\,d\phi(u) +2\psi(i)\,d\phi(u)]\\ 
& =  \phantom{-}\,\psi(i)\Delta\phi(i) - \nabla\psi\negthinspace\cdot\negthinspace\nabla\phi(i).  
\end{align*}
Thus,
\begin{align*}
\sum_{j\in V_H}[\psi(j)\Delta\phi(j)-\nabla\psi\cdot\nabla\phi(j)] & =\sum_{j\in V_H}\Div X(j)\\
& =  \sum_{j\in V_{\partial H}} n_H\negthickspace\cdot\negthickspace X(j)\\
&  =\sum_{j\in V_{\partial H}} n_H\negthickspace\cdot\negthinspace\psi\nabla\phi(j), 
\end{align*}
and this proves the first identity. 

\mpar
Using Green's first identity applied to $X=\psi\nabla\phi$ and $Y=\phi\nabla\psi$ then subtracting the results gives the second identity. 

\mpar
Finally, the third identity follows from the second if we use the fact that, 
$$
\Delta G_i (j) = e_i(j)-\frac {1}{|V_G|}.
$$
and choose $\psi = G_i$.
\end{proof}

\begin{remark}
1. Define an operator $G$ by the formula $G\phi(j)=\sum_{i\in V_G}\phi(i)G_i(j)$. Then it follows from the definitions that $G\colon\negthinspace\mathring{C}(G)\to\negthinspace\mathring{C}(G)$ and $\Delta G\phi = G\Delta\phi = \phi$. This is just to say that the numerical values of Green's function are the matrix elements of $\Delta^{-1}$, namely, $\Delta^{-1}(i,j)=G_j(i).$

\mpar
2. Let $\Delta^H$ be the restriction of $\Delta$ to $C(H)$ and for $\phi\in C(H)$ let $G^H\phi =G(1_H\phi)$. Then Green's third identity says that $G^H$ almost inverts $\Delta^H$ on $\mathring{C}(G)$ in the sense that $G^H\Delta^H\phi =\phi$ plus an error term that is expressed as a sum over the boundary $\partial H$.

\mpar
3. Note that $\Delta^H$ is not the Laplacian $\Delta_H$ of $H$ in general, but it is so if further conditions are imposed on $E_H$ and $\phi$. First, assume that $H$ is the induced subgraph of $V_H$ meaning every edge of $G$ having both endpoints in $H$ is also an edge of $H$. Second, assume $\phi$ satisfies Neumann boundary conditions $\phi(\pi_+(u))=\phi(\pi(u))$ for all $u\in V_{t\partial H}$ or, equivalently, $d\phi(u)=0.$ Under these conditions we have,
$$
\{u\in V_{tH}\mid \pi(u)\in V_H\setminus V_{\partial H^-}\} = \{u\in V_{tG}\mid\pi(u)\in V_H\setminus V_{\partial H^-}\},
$$
hence $\Delta_H\phi(i)=\Delta\phi(i)$ for all interior points $i\in V_H\setminus V_{\partial H^-}$. On the other hand, if $i\in V_{\partial H^-}$ is a boundary point then, 
$$
\{u\in V_{tG}\mid \pi(u)=i\} = \{u\in V_{tH} \mid\pi(u)=i\}\cup \{u\in V_{t\partial H}\mid \pi(u)=i\},
$$
hence $\Delta_H\phi(i)=\Delta\phi(i)$ because $d\phi(u)$ vanishes if $u\in V_{t\partial H}$. Observing that $C(H)$ is isomorphic to the subspace of $C(H\cup\partial H)$ satisfying Neumann boundary conditions, we observe that  Green's theorem verifies that $\Delta^H\colon\mathring{C}(H)\to\mathring{C}(H)$ is an isomorphism and that Green's third identity states that $G^H$ almost inverts $\Delta_H$ on $\mathring{C}(G).$

\mpar
4. These observations were important historically because there are explicit formulas for Green's function in $\mathbb{R}^d.$ Think of the restriction of Green's function to a domain $D$ as defining an integral operator. The fact that this operator almost inverts the Laplacian on $D$ effectively reduced the solution of the Dirichlet problem in $D$ to the solution of more tractable integral equations on the boundary $\partial D.$ This doesn't seem particularly valuable in the case of graphs since there are now fast Monte Carlo approximations for solving boundary value problems on meshes and networks. However, it is interesting that tangent graphs have a sufficiently rich intrinsic geometric structure to support analogues of these classical results. 
\end{remark}

We end this section with Helmholtz's theorem on the decomposition of vector fields which we'll elaborate upon in the next section.

\begin{proposition} (Helmholtz's Theorem) Every vector field on $G$ can be uniquely written as the sum of a gradient field and a divergence-free field. More precisely,
$$
\mathcal{X}(G) = \Image (\nabla) \oplus \Ker (\Div)
$$
is an othogonal decomposition.
\end{proposition}

\begin{proof}
Let $p_{\nabla}\colon\mathcal{X}(C)\to\mathcal{X}(G)$ be defined by
$$
p_{\nabla} = \nabla\circ\Delta^{-1}\circ\Div
$$
and observe that it is a projection,
\begin{align*}p_{\nabla}^2 = & (\nabla\circ\Delta^{-1}\circ\Div ) (\nabla\circ\Delta^{-1}\circ\Div)\\= &   \nabla\circ\Delta^{-1}\circ\Delta\circ\Delta^{-1}\circ\Div\\= & \nabla\circ\Delta^{-1}\circ\Div = p_{\nabla}.\end{align*}
Since $\nabla\circ\Delta^{-1}\colon\mathring{C}(G)\to\Image{\nabla}$ is an isomorphism it follows from Proposition 2.8.2 that,
$$
\Image (p_{\nabla})=\Image (\nabla) \,\,\text{and}\,\, \Ker(p_{\nabla})=\Ker(\Div)=\Image(\nabla)^{\perp},
$$
Thus, $\Image(\nabla)$ and $\Ker(\Div)$ are complementry orthogonal subspaces of $\mathcal{X}(G).$
\end{proof}

\section{Curl of a Vector Field}

It's easy to see that if $s\colon\mathcal{X}(G)\to\mathcal{X}(C)$ is the projection onto $\mathcal{X}^s(C),$ the space of symmetric vector fields, then $\Image(\nabla)\subset\Ker(s)$ and $\Image(s)\subset\Ker(\Div).$ Thus,
$$
0\xrightarrow{}\mathring{C}(G)\xrightarrow{\nabla}\mathcal{X}(G)\xrightarrow{s}\mathcal{X}(G)\xrightarrow{\Div}\mathring{C}(G)\xrightarrow{} 0
$$
is an exact sequence. The formal similarity of this sequence with the relations among the gradient, divergence, and curl of vector fields in $\mathbb{R}^3$ leads some authors to regard the projection $s$ as the curl operator on vector fields on a graph. However, there are other interesting operators that form an exact sequence with $\nabla$ and $\Div$ as above, not just $s.$ Motivated by the classical case, and by \cite{S}, we propose a new and geometrically meaningful operator as the definition of the curl.

\mpar
In $\mathbb{R}^3,$ the component of $\Curl X$ at point $x$ in direction $u$, where $u$ is a unit vector, is computed in the following way. It is the limit of the line intergral of $X$ around the boundary of a small surface in the plane perpendicular to $u$ containing $x,$ divided by the area of the surface, as the surface shrinks to a point. While this definition can't be extended to graphs, it suggests a way forward by asking $\Curl X$ to have the same line integrals as $X$ around certain closed walks in $G$.

\mpar

Let's begin by looking at a consequence of Helmholtz's theorem. If $X\in\mathcal{X}(G)$ and $Y = (1-p_{\nabla}) X$ then $Y$ has two properties: it is divergence-free and it has the same line integrals as $X$ along any closed path. The first property holds because $\Image (1-p_{\nabla})=\Ker (\Div)$ and the second property holds because $X-Y$ is the gradient of a function and gradient fields have vanishing line integrals along closed walks. 

\mpar
The analogy with $\Curl$ in $\mathbb{R}^3$ suggests we only require that $X$ and $Y$ have the same line integrals around simple cycles; that is, subgraphs of $G$ isomorphic to $C_n$ for some $n\geq 3$. (Note that we explicity exclude closed walks of length two that traverse a single edge). In doing so, we would be justified in saying that $Y$ has the same circulation as $X$. But the two requirements expected of a geometrically meaningful notion of $\Curl X$ - that $Y$ has zero divergence and the same circulation as $X$ - are not enough to specify $Y$ uniquely. This is because there exist nonzero vector fields that are both divergence-free and circulation-free. However, there is a unique such $Y$ that is orthogonal to the space of divergence-free and circulation-free fields and this suggests the definition of curl.

\mpar
To make sense of these observations let us be precise about the terms of argument.

\begin{definition}
1. A \textit{walk} $\omega$ in $G$ is a sequence of mutually adjacent vertices: $\{\omega_{n-1}, \omega_n\}\in E_G$ for all $1\leq n\leq N$ and $N\geq 3.$ We write the walk as a word $\omega=\omega_0\,\omega_1\,\omega_2\cdots\omega_N$ where $N=N(\omega)\geq 1$ is the \textit{length} of the walk.

\mpar
2. A \textit{trail} is a walk with no repeated edges: $\{\omega_n, \omega_{n+1}\}\neq \{\omega_m, \omega_{m+1}\}$ for $1\leq n\neq m\leq N-1.$

\mpar
3. A \textit{circuit} is a closed trail: $\omega_N=\omega_0.$ 

\mpar
4. A \textit{simple circuit} is a circuit with no repeated vertices: $\omega_n\neq\omega_m$ for $1\leq n\neq m\leq N-1.$

\mpar
5. The \textit{support} of a walk $\omega$ is the graph $\Spt\omega$ whose vertex set is $V_{\Spt\omega}=\{\omega_n\mid 0\leq n\leq N-1\}$ and whose edge set is $E_{\Spt\omega}=\{\,\{\omega_n, \omega_{n+1}\}\mid 0\leq n\leq N-1\}.$

\mpar
6. A graph $H$ \textit{supports} a walk $\omega$ provided $\Spt\omega$ is a subgraph of $H.$

\mpar
7. A graph is a \index{cycle}\textit{cycle} provided it is the support of a circuit. 
\mpar

8. A graph is a $\textit{simple cycle}$ provided it is the support of a simple circuit.

\mpar
9. The \textit{line integral} of $X$ along a walk $\omega$ is the sum,
$$
\omega\negthinspace\cdot\negthinspace X=\sum_{n=1}^{N(\omega)}X(\omega_{n-1}\omega_n).
$$

10. The \textit{circulation} of $X$ around a circuit $\omega$ is the line integral $\omega\negthinspace\cdot\negthinspace X.$

\mpar
11. A vector field is said to be \textit{circulation-free} provided $\omega\negthinspace\cdot\negthinspace X=0$ for every simple circuit. The space of circulation-free vector fields is denoted $\mathcal{Z}(G).$

\mpar
12. A vector field is said to be \textit{harmonic} provided it is divergence-free and circulation-free. The space of all harmonic vector fields is denoted $\mathcal{H}(G)$

\end{definition}

\begin{remark}
1. We emphasize the disctiction between circuits and cycles. A cycle is static - it is a graph. A circuit is dynamic - it is an ordered sequence of vertices of a cycle. In general, a given cycle is the support of many circuits each having different starting points and orderings.  A circuit may have repeated vertices but it has no repeated edges. A simple cycle is has neither repeated vertices nor repeated edges and therefore is isomorphic to $C_n$ for some $n\geq 3.$ It is the support of $2n$ distinct simple circuits, since circuits have two possible orderings for each possible starting point.

\mpar
2. We prefer to use the direct definition of line integral of a vector field along a walk, but it's interesting to note that for trails there is an equivalent definition that is analogous to line integrals along smooth curves in $\mathbb{R}^n.$ Observe that the support of a trail $\omega$ is a subgraph of $G$ and $\omega$ determines a unique ordering of the vertices of $\text{spt}(\omega).$ While a vertex may be repeated in a trail, its edges are not. This means each edge of $\text{spt}(\omega)$ is traversed exactly once by $\omega,$ a fact that we use to define the tangent vector field $t_{\omega}$ of the trail. Specifically, if $\omega_k=i$ and $\omega_{k+1}=j$ are consecutive vertices along the trail then $ij$ is a vertex of the tangent graph of $\text{spt}(\omega)$ and the tangent vector field is defined by the rule $t_{\omega}(ij)=1$ and $t_{\omega}(u)=0$ if $\pi(u)=i$ and $u\neq ij.$ In this notation we have,
$$
\omega\negthinspace\cdot\negthinspace X =\sum_{k=1}^{N(\omega)} t_{\omega}\negthinspace\cdot\negthinspace X(\omega_k),
$$
which justifies the line integral terminology.

\mpar
3. Evidently every walk defines a linear functional on $\mathcal{X}(G),$ that is to say a $1$-form in $\mathcal{X}(G)^*.$
\end{remark}

At this point we have three orthogonal decompositions,
$$
\mathcal{X}(G)=\mathcal{X}^s(G)\oplus\mathcal{X}^a(G)=\Image(\nabla)\oplus\Ker(\Div)= \mathcal{Z}(G)\oplus\mathcal{Z}(G)^{\perp}
$$
and it's helpful to explore some of the basic relations among them through examples.

\begin{example}
(Trees) If $G$ is a tree then every vector field is circulation-free since $G$ supports no cycles. Thus $\mathcal{Z}(G)=\mathcal{X}(G) = \mathcal{X}^s(G)\oplus\mathcal{X}^a(G).$ It's easy to see that $\mathcal{X}^a(G)=\Image(\nabla)$ simply by choosing a root vertex $i$ arbitrarily, defining $\phi(j)=\omega\cdot X$ where $\omega$ is the unique shortest walk in $G$ from $i$ to $j$, and observing that $\nabla\phi=X.$ It just as easy to see that $\mathcal{X}^s(G)=\Image(\nabla)^{\perp}=\Ker(\Div)=\mathcal{H}(G).$ It follows that $\mathcal{Z}(G)=\Image(\nabla)\oplus\mathcal{H}(G)$ and $|\Image(\nabla)|=|\mathcal{H}(G)|=|V_G|-1.$
\end{example}

\begin{example}
(Unicyclic Graphs) A unicyclic graph $G$ is a graph having cyclomatic number $|E_G|-|V_G|+1 =1.$ Thus $G$ contains a unique cycle and the complement of that cycle is a forest of trees rooted in the cycle. Let's begin by looking at $C_n,$ the cycle with with $n$ edges joining $n$ vertices labelled $1,2, \dots, n$. Clearly, $|\mathcal{X}(G)|=2n$ and $X\in \mathcal{Z}(G)$ if and only if,
$$
X(12)+X(23)+\cdots X(n1)= X(1n)+X(n(n-1))+\cdots X(21)=0.
$$
We have $|\mathcal{Z}(G)|=2n-2$ and $|\mathcal{Z}(G)^{\perp}|=2$ since the two equations above are linearly independent. Adding and subtracting them we find,
$$
sX(12)+sX(23)+\cdots sX(n1)= aX(12)+aX(23)+\cdots aX(n1)=0,
$$
hence $|\mathcal{Z}(G)\cap\mathcal{X}^s(G)|= |\mathcal{Z}(G)\cap\mathcal{X}^a(G)|=n-1.$ Thus, there is a one dimensional subspace of antisymmetric, non-circulation-free vector fields and a one dimensional subspace of symmetric, non-circulation-free vector fields. By trial and error one finds the vector field with coefficients $Y^a(i(i+1)) =1,\, Y^a((i+1)i)=-1$ generates the former subspace and the constant vector field $Y^s(u)=1$ generates the latter. 

\mpar
Let's show that $Y^a$ and $Y^s$ span $\mathcal{Z}^\perp.$ Suppose $X\in\mathcal{Z}(G)$ and observe that,
\begin{align*}
\langle X, Y^a\rangle_{\mathcal{X}(G)} & =\sum_{u\in tG}X(u)Y^a(u)\\
&= g\,\sum_{i=1}^{n-1}X(i(i+1)) + g\, X(n1) - g\,\sum_{i=1}^{n-1} X((i-1)i) -g\, X(1n)\\
& = g\,(\omega\negthinspace\cdot\negthinspace X - \overline{\omega}\negthinspace\cdot\negthinspace X),
\end{align*}
where $\omega = 1 2 3\cdots n 1$ is the natural circuit on $G$ and $\overline{\omega}$ is the reverse circuit. Since $X$ is circulation-free, both terms above vanish hence $Y^a$ is orthogonal to $\mathcal{Z}(G).$ A similar argument shows $Y^s$ is also orthogonal to $X,$ so $Y^a, Y^s\in\mathcal{Z}(G)^\perp.$ But $Y^a$ and $Y^s$ are linearly independent so they span $\mathcal{Z}(G)^\perp.$ Observe that, 
$$
\Image(\nabla)\subset\mathcal{Z}(G)\cap\mathcal{X}^a(G)
$$ 
and both spaces have dimension $n-1$ so they coincide. Also observe that,
$$
\mathcal{Z}(G)\cap\mathcal{X}(G)^s\subset\mathcal{H}(G),
$$ 
since all symmetric vector fields have zero divergence, hence $|\mathcal{H}(G)|\geq n-1.$

\mpar
Thus, we have $|\Image(\nabla)|=n-1,\, |\mathcal{Z}(G)^\perp|=2,\,\,\text{and}\,\, |\mathcal{H}(G)|\geq n-1.$ But, 
$$
|\Image(\nabla)|+|\mathcal{Z}(C)^\perp|+|\mathcal{H}(G)|\leq |\mathcal{X}(G)| =2n,
$$
and therefore $|\mathcal{H}(G)|=n-1.$ 

\mpar
Now suppose $G$ is a forest of trees rooted in a cycle $C_G.$ The edges of these trees play no role in determining whether a vector field is circulation-free so we still have $|\mathcal{Z}(G)^{\perp}|=2$ and 
$$
|\mathcal{Z}(G)|=2|E_G|-2= 2|E_G|-2(|E_G|-|V_G|+1)=2(|V_G|-1).
$$
A moment's thought reveals the extention of $X\in\mathcal{Z}(C_G)^\perp$ defined by setting the coefficients of $X$ equal to zero on the directed edges of the pendant trees, is orthogonal to circulation-free vector fields on $G.$ This implies the extention of $\mathcal{Z}(C_G)^\perp$ equals $\mathcal{Z}(G)^\perp$ because they have the same dimension. 

\mpar
Similarly, the extension of $X\in\mathcal{H}(C_G)$ is a harmonic vector field in $\mathcal{H}(G)$ and it remains harmonic if the coefficients of the directed edges of the pendant trees are non-zero and symmetric. Thus, $|\mathcal{H}(G)|\geq |V_G|-1.$ But $|\Image(\nabla)|=|V_G|-1$ so $|\mathcal{H}(G)|= |V_G|-1,$ as well. Thus, $\mathcal{X}(G)=\Image(\nabla)\oplus\mathcal{Z}(G)^{\perp}\oplus\mathcal{H}(G),$
which completes the unicyclic case.
\end{example} 

It follows from the definitions that for general graphs, $\Image(\nabla)\subset\mathcal{Z}(G)\cap\mathcal{X}^a(G)$ and $\mathcal{H}(G)\supset\mathcal{Z}(G)\cap\mathcal{X}^s(G).$ We used this fact in the previous example. Equality holds for trees and unicyclic graphs but the next example shows this is not always the case.

\begin{example}
Let $G$ be a rectangle having edges $\{1,2,3,4\}$  with an added edge between a pair of diagonal vertices say, $1$ and $3$. (See Figure 3.) Note that $G$ has cyclomatic number $|E_G|- |V_G| +1 = 2$ and $|\mathcal{X}(G)|=10.$ Circulation-free fields are solutions of six simultaneous equations in ten variables,
\begin{align*}
X(12) +X(23)+X(34)+X(41) &=0\\
X(14) +X(43)+X(32)+X(21) &=0\\
X(12)+X(23)+X(31) & = 0\\
X(13)+X(32)+X(21) & =0\\ 
X(14)+X(43)+X(31) & = 0\\
X(13)+X(34)+X(41) & =0.
\end{align*}
The system has rank four because there are two relations between the circulations around the the outer square and the inner triangles. Specifically, 
\begin{align*}
X(12)+X(23) & =-X(31) = - (X(34)+X(41))\\
X(32)+X(21) & =-X(13) = -(X(14)+X(43)),
\end{align*}
hence $|\mathcal{Z}(G)^{\perp}|=4$ and $|\mathcal{Z}(G)|=6.$

\begin{figure}[h]
\begin{tikzpicture} 
\draw[fill=black] (0,0) circle (2pt);
\draw[fill=black] (0,2) circle (2pt);
\draw[fill=black] (2,0) circle (2pt);
\draw[fill=black] (2,2) circle (2pt);

\node at (0,-0.3) {1};
\node at (0, 2.3) {4};
\node at (2,-0.3) {2};
\node at (2,2.3) {3};

\draw[thick] (0,0) -- (2,0);
\draw[thick] (2,0) -- (2,2) -- (0,2) -- (0,0);
\draw[thick] (0,0) -- (2,2);

\node at (1,-0.3) {a};
\node at (-0.3, 1) {d};
\node at (1,2.3) {c};
\node at (2.3,1) {b};
\node at (1.3,1) {e};

\draw[fill=black] (5,0) circle (2pt);
\draw[fill=black] (5,2) circle (2pt);
\draw[fill=black] (7,0) circle (2pt);
\draw[fill=black] (7,2) circle (2pt);

\node at (5,-0.3) {1};
\node at (5, 2.3) {4};
\node at (7,-0.3) {2};
\node at (7,2.3) {3};

\draw[thick] (5,0) -- (7,0);
\draw[thick] (5,0) -- (5,2) -- (7,2) -- (7,0);
\draw[thick] (5,0) -- (7,2);

\draw[thick,->] (5,0) -- (6,0);
\draw[thick,->] (7,0) -- (7,1);
\draw[thick,->] (7,2) -- (6,2);
\draw[thick,->] (5,2) -- (5,1);
\draw[thick,->] (7,2) -- (6,1);

\node at (6,-0.4) {f};
\node at (4.6, 1) {-f};
\node at (6,2.4) {-f};
\node at (7.4,1) {f};
\node at (6.2,0.8) {-2f};
\end{tikzpicture}
\caption{(\textit{Left}) The graph $G$ with edges labelled by the coefficients of a generic even vector field. (\textit{Right}) The graph $G$ with edges labelled by the coeffieicents of a specific antisymmetric vector field.} 
\end{figure}
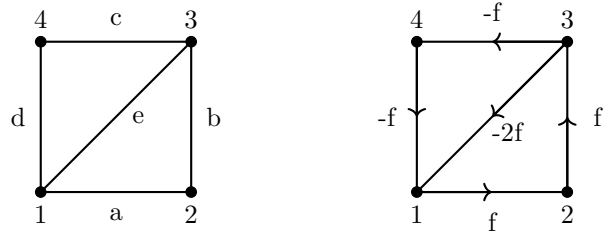

\mpar
Let's look for solutions by writing the circulation equations for even vector fields in the form $a+b+c+d = a+b+e = c+d+e=0.$ By inspection the solutions are $a=-b, c=-d, e=0,$ which yields a two dimensional space of harmonic vector fields in $\mathcal{Z}(G)$. On the other hand, there is a three dimensional space of gradient fields in $\mathcal{Z}(G)$ so there remains one degree of freedom yet to be identified in the space of circulation-free fields. It can't correspond to a symmetric field or a gradient field so it makes sense to search for a divergence-free, anti-symmetric field, say $Y$. By trial and error one finds,
\begin{align*}
f & =Y(12)=Y(23)\\
-f & = Y(34)=Y(41)\\
-2f & = Y(31).
\end{align*}
(See Figure 3). By inspection $Y$ is  circulation-free and divergence-free and therefore harmonic. Thus, $\mathcal{X}(G)=\Image(\nabla)\oplus\mathcal{Z}(G)^{\perp}\oplus\mathcal{H}(G)$ where,
$$
|\Image(\nabla)| =|\mathcal{H}(G)|= 3 = |V_G|-1\,\,\, \text{and}\,\,\, |\mathcal{Z}(G)^{\perp}|=4 = 2(|E_G|-|V_G|-1).
$$
\end{example}

The next result examines the role of parity of vector fields in more detail. Item 1 is well known but item 2 is new.

\begin{proposition}
1. The sequence,
$$
0\xrightarrow{}\mathring{C}(G)\xrightarrow{\nabla}\mathcal{X}(G)\xrightarrow{\,s\,}\mathcal{X}(G)\xrightarrow{\Div}\mathring{C}(G)\xrightarrow{} 0
$$
is exact and the homology groups $\mathcal{X}^a(G)/\Image(\nabla)\cong\Ker(\Div)/\mathcal{X}^s(G)$ have dimension $ |E_G|-|V_G|+1.$

\mpar
2. Let $\mathcal{Z}^s(G), \mathcal{Z}^a(G)$ and $\mathcal{H}^s(G), \mathcal{H}^a(G)$ be the images of the parity projections restricted to circulation-free and harmonic vector fields, respectively. We have, 
$$
\mathcal{Z}(G)=\mathcal{Z}^s(G)\oplus\mathcal{Z}^a(G) \,\,\text{and}\,\, \mathcal{H}(G)=\mathcal{H}^s(G)\oplus\mathcal{H}^a(G).
$$
\end{proposition}

\begin{proof}
The sequence is exact since $\Ker(s)=\mathcal{X}^a(G)\supset\Image(\nabla)$ and $\Image(s)=\mathcal{X}^s(G)\subset\Ker(\Div).$ Since $|\mathcal{X}^s(G)|=|\mathcal{X}^a(G)| =|E_G|, |\Image(\nabla)|=|V_G|-1,$ and $|\Ker(\Div)|=2|E_G|-|V_G|+1$, the dimension of the homology groups equals the cyclomatic number of $G$.

\mpar
Let $\omega$ be a simple circuit and let $\overline{\omega}$ be the circuit with the same starting point as $\omega$ but opposite orientation. Then $\omega\negthinspace\cdot\negthickspace X=\omega\negthinspace\cdot\negthickspace sX+\omega\negthinspace\cdot\negthickspace aX=0$ for any circulation-free vector field, hence $\omega\negthinspace\cdot\negthickspace sX=-\omega\negthinspace\cdot\negthickspace aX.$ On the other hand, 
$$
-\omega\negthinspace\cdot\negthickspace aX=\overline{\omega}\negthinspace\cdot\negthinspace aX=-\overline{\omega}\negthinspace\cdot\negthinspace sX=-\omega\negthinspace\cdot\negthinspace sX  =\omega\negthinspace\cdot\negthinspace aX.
$$
It follows that $\omega\negthinspace\cdot\negthinspace aX = 0 = \omega\negthinspace\cdot\negthinspace sX.$ Therefore, $X\in \mathcal{H}(G)$ implies $sX$ and $aX$ are circulation-free. However, since $sX$ is symmetric, it is divergence-free and therefore harmonic, as is $aX=X-sX.$
\end{proof}

\begin{remark}
1. This result says the extent to which asymmetric vector fields fail to be gradients equals the extent to which divergence-free vector fields fail to be symmetric and that this failure is a topological property of the graph, measured by its cyclomatic number. 

\mpar
2. The decomposition  $\mathcal{H}(G)=\mathcal{H}^s(G)\oplus\mathcal{H}^a(G)$ suggests trying to identify a basis for harmonic vectors fields by searching for symmetric, circulation-free fields and finding the remaining basis elements among anti-symmetric, divergence-free vector fields, as in Example 3.5. 

\mpar
3. It would be interesting to find a formula for the dimensions of $\mathcal{H}^s(G)$ and $\mathcal{H}^a(G).$ 
\end{remark}

\begin{definition}
The \textit{curl} operator is the orthogonal projection of $\mathcal{X}(G)$ onto $\mathcal{Z}(G)^{\perp}.$
\end{definition}

\begin{proposition}
Let $X\in\mathcal{X}(G).$ Then,
\begin{align}
&\,\omega\cdot X = \omega\cdot\Curl X\,\,\text{for all simple circuits $\omega$}\\
&\Div\Curl X =0\\
&\Curl X \in \mathcal{H}(G)^{\perp}.
\end{align}
\end{proposition}

\begin{proof}
Since $\Image(\nabla)\oplus\mathcal{H}(G)\subset\mathcal{Z}(G)$ it follows that $\Curl X$ is orthogonal to gradient fields and harmonic fields. Thus, $\Curl X\in\mathcal{H}(G)^{\perp}$ and $\Curl X\in\Image(\nabla)^{\perp}=\Ker(\Div)$ which proves items 2 and 3. Observe that $X-\Curl X$ is circulation-free since $\Curl(X-\Curl X)= \Curl X -\Curl^2 X = 0,$ which proves item 1.
\end{proof}

The Helmholtz-Hodge decomposition shows that these three properties actually define the curl. The argument turns on an elementary result in linear algebra whose proof we include for the reader's convenience.
 
 \begin{lemma}
 Let $A, B$ and $C$ be finite dimensional inner product spaces and let $f$ and $g$ be linear transformations such that $A \xrightarrow{f} B\xrightarrow{g} C$ and $g\circ f = 0.$ Then 
 $$
 B = \Image(f)\oplus\Image(g^*)\oplus\Ker(ff^*+g^*g),
$$
 is an orthogonal decomposition. Furthermore, $\Ker(g)/\Image(f)\cong \Ker(ff^*+g^*g) =\Ker(f^*)\cap\Ker(g).$
 \end{lemma}
 
 \begin{proof}
 In self evident notation we have $\langle fa, g^*c\rangle_B = \langle g\circ fa, c\rangle_C=0,$ hence $\Image(f)$ and $\Image(g^*)$ are orthogonal subspaces of $B$. On the other hand,
$$
(\Image(f)\oplus\Image(g^*))^{\perp} = \Image(f)^{\perp}\cap\Image(g^*)^{\perp} = \Ker(f^*)\cap\Ker(g).
$$
 Now, observe that,
 $$
 \langle b, (f f^* + g^* g)b\rangle_B = \langle f^* b, f^* b\rangle_A +  \langle gb, gb\rangle_C = |f^*x|^2_A + |gx|^2_C  
 $$
 and therefore $b\in\Ker (f f^* + g^* g)$ if and only if $b\in\Ker (f^*)\cap\Ker (g).$ This establishes the direct sum decomposition of $B$ from which we conclude that $\Ker(g)/\Image(f)\cong\Image(g^*)^{\perp}/\Image(f)\cong\Ker(ff^*+g^*g).$
 \end{proof}
 
\begin{remark}
This lemma is useful because it identifies the homology group as the kernel of the self-adjoint operator $f f^* + g^* g,$ sometimes called the Hodge operator or Hodge laplacian in honor of W. V. D. Hodge.
\end{remark}

\begin{theorem}
(Helmholtz-Hodge Decomposition) The sequence,
$$
0\xrightarrow{} \mathring{C}(G)\xrightarrow{\nabla}\mathcal{X}(G)\xrightarrow{\Curl}\mathcal{X}(G)\xrightarrow{\Div}\mathring{C}(G)\xrightarrow{}0
$$
is exact and $\mathcal{X}(G)=\Image(\nabla)\oplus\Image(\Curl)\oplus\mathcal{H}(G)$ is an orthogonal decomposition. The summands have dimension,
$$
|\Image(\nabla)|=|\mathcal{H}(G)|= |V_G|-1
$$
and, 
$$
|\Image(\Curl)|=2(|E_G|-|V_G)| -1).
$$
\end{theorem}

\begin{proof}
Observe that $\Image(\nabla)\subset\mathcal{Z}(G)=\Ker(\Curl)$ and $\Image(\Curl)\subset\Image(\nabla)^{\perp}=\Ker(\Div)$ hence the sequence is exact. Lemma 4.10 applied either to $f=\nabla, g=\Curl$ or $f=\Curl, g=\Div$ yields the decomposition, using the fact that $\Curl$ is self adjoint.

\mpar
We calculate the dimension of $\mathcal{H}(G)$ and $\Image(\Curl)=\mathcal{Z}(G)^{\perp}$ by induction on the cyclomatic number $\xi=|E_G|-V_G+1.$ The cases $\xi=0,1$ were established in Examples 4.3 and 4.4.

\mpar
So, suppose the result is true for all graphs with $\xi=n$ and let $G$ be a graph with $\xi=n+1.$ Let $T$ be a spanning tree of $G$ and consider the $n+1$ simple cycles labelled by the $n+1$ edges of $G$ not in $T$. Let $\{i,j\}$ be one such edge and let $G^{\prime}$ be the subgraph of $G$ obtained by deleting $\{ i,j \}$; specifically,
$$
V_{G'} = V_G\ \text{and}\ E_{G'} = E_G\setminus\{ i,j\}.
$$
Then $G'$ has cyclomatic number $n$ and so by the inductive hypothesis,
 $$
 |\mathcal{H}(G' )|=|V_{G'}| -1 \ \text{and}\,\, |\mathcal{Z}(G')^{\perp}| = 2n.
 $$
For any $X\in\mathcal{X}(G)$ let its restriction to $G'$ be denoted $X'\in\mathcal{X}(G'),$ where $X'(u)=X(u)$ for all $u\in V_{tG'}.$ Clearly, if $X\in \mathcal{Z}(G)$ then $X'\in\mathcal{Z}(G')$ because every cycle in $G^\prime$ is a cycle in $G.$ On the other hand, suppose $Y'\in\mathcal{Z}(G')$ and define its extension $Y\in\mathcal{Z}(G)$ by the following rule. Let $\omega$ be a simple circuit on the simple cycle of $G$ labelled by $\{ i,j \}$ and suppose, without loss of generality, that $\omega_0 = \omega_N = i$ and $\omega_1=j.$ Let $Y(u) = Y'(u)$ for all $u\in V_{tG'}$ and set,
$$ 
Y(ij)= -\sum_{n=1}^{N-1} Y'(\omega_n\omega_{n+1}) \,\, \text{and} \,\, Y(ji)= -\sum_{n=1}^{N-1} Y'(\omega_{n+1}\omega_n).
$$
Then $Y\in\mathcal{Z}(G)$ and the restriction of $Y$ to $G'$ is $Y'$. Therefore, $|\mathcal{Z}(G)| = |\mathcal{Z}(G')|$ and,
 \begin{align*}
 |\mathcal{Z}(G)^{\perp}| & = |\mathcal{X}(G)| - |\mathcal{Z}(G)|\\
 & =  |\mathcal{X}(G')| + 2 - |\mathcal{Z}(G')|\\ 
 & = |\mathcal{Z}(G')^{\perp}| + 2\\
 & = 2(n+1)\\
 & = 2(|E_G|-|V_G| +1).   
 \end{align*}

Finally, since $|V_G| = |V_{G'}|,$ it follows  that $|\mathcal{H}(G)| = |\mathcal{H}(G')|=|V_G| -1.$
\end{proof}

\begin{remark}
1. In three dimensions, $\Curl X =\nabla\times X$ is a local operator. Previous definitions of curl on a graph, like that in \cite{L} relative to the clique complex, or in \cite{BCEG} as taking the symmetric part of a vector field, are also local in an appropriate sense. In contrast, our version of $\Curl X$ is a non-local operator since it is defined in terms of the solutions of a set of homogeneous linear equations indexed by simple cycles of the graph. While this may be surprising because it differs so markedly from the classical case, it's not unusual in that the Helmholtz projections $p_{\nabla}$ and $1-p_{\nabla}$ are non-local operators. It would be interesting to find a formula for the $\Curl$ in terms of primitive operators, in the same sense that $p_{\nabla}=\nabla\circ\Delta^{-1}\circ\Div$ is a formula  for the projection of $\mathcal{X}(G)$ onto $\Image(\nabla).$ Note that this is equivalent to finding a formula for the orthogonal projection onto $\mathcal{H}(G).$

\mpar
2. We know $\Image(\nabla)\cong\mathcal{H}(G),$ since they have the same dimension. This begs the question whether there is a geometrically meaningful isomorphism between them. The answer is not obvious since $p_{\nabla}X=0$ if $X$ is harmonic. Moreover, gradient fields are antisymmetric but harmonic fields may have mixed parity and so this question is likely related to the previous question of a formula for the dimensions of $\mathcal{H}^s(G)$ and $\mathcal{H}^a(G).$ Similar remarks apply to whether there is a geometrically meaningful isomorphism between $\Image(\Curl)$ and $\mathcal{X}^a(G)/\Image(\nabla)\oplus\Ker(\Div)/\mathcal{X}^s(G).$ 

\mpar
3. With the curl operator in hand, we can formulate Maxwell's equations on a graph as,
$$
\tfrac{\partial}{\partial t} E_t=-\Curl B_t,\,\,\tfrac{\partial}{\partial t} B_t =-J+\Curl E_t,\,\,\Div E_t=\rho, \,\, \Div B_t=0,
$$
where $E$ is the electric field, $B$ is the magnetic field, $\rho$ is the electric charge density, and $J$ is the current density. This is a constrained system of linear ordinary differential equations. It would be interesting to know if this form of Maxwell's equations is just a mathematical curiosity or if it has genuine physical meaning; for example, if the graph represents an electromagnetic device like a network of waveguides or if it represents a coarse-grained description of a device as a system of lumped circuits. 
\end{remark}

\end{document}